\documentclass{amsart}
\usepackage{commands}
\usepackage[utf8]{inputenc}

\newcommand{\CS}{\text{CS}}
\newcommand{\Tor}{\text{Tor}}
\newcommand{\Tr}{\text{Tr}}
\newcommand{\Ima}{\text{Im}}
\newcommand{\Kernel}{\text{Ker}}

\newcommand{\Sym}{\text{Sym}}

\newcommand{\tq}{\tilde{q}}
\newcommand{\SU}{\text{SU}}
\newcommand{\bbZ}{\mathbb{Z}}
\newcommand{\bbR}{\mathbb{R}}
\newcommand{\bbC}{\mathbb{C}}

\newcommand{\Hom}{\textrm{Hom}}

\newcommand{\SL}{\textrm{SL}}

\newcommand{\Adj}{\textrm{Adj}}

\theoremstyle{definition}
\newtheorem{theorem}{Theorem}[section]

\newtheorem{lemma}[theorem]{Lemma}
\newtheorem{lemma-definition}[theorem]{Lemma-Definition}

\newtheorem{definition}[theorem]{Definition}

\newtheorem{proof}{Proof}[section]
\newtheorem{theorem*}{Theorem}[section]

\title{From Torus Bundles to Particle-Hole Equivariantization}
\author{Shawn X. Cui}
\address{Department of Mathematics and Department of Physics and Astronomy\\Purdue University\\
West Lafayette\\
IN 47907\\U.S.A.}
\email{cui177@purdue.edu}

\author{Paul Gustafson}
\address{Department of Electrical Engineering\\
Wright State University\\
Dayton\\
OH 45435\\
U.S.A
}
\email{paul.gustafson@wright.edu}

\author{Yang Qiu}
\address{Department of Mathematics\\
University of California\\Santa Barbara\\CA 93106\\U.S.A}
\email{yangqiu@math.ucsb.edu}

\author{Qing Zhang}
\address{Department of Mathematics\\Purdue University\\
West Lafayette\\
IN 47907\\U.S.A.}
\email{zhan4169@purdue.edu}
\date{}

\begin{document}
\maketitle
\begin{abstract}
We continue the program of constructing (pre)modular tensor categories from 3-manifolds first initiated by Cho-Gang-Kim using $M$ theory in physics and then mathematically studied by Cui-Qiu-Wang. An important structure involved in the construction is a collection of certain $\text{SL}(2, \mathbb{C})$ characters on a given manifold which serve as the simple object types in the corresponding category. Chern-Simons invariants and adjoint Reidemeister torsions also play a key role, and they are related to topological twists and quantum dimensions, respectively, of simple objects. The modular $S$-matrix is computed from local operators and follows a trial-and-error procedure. It is currently unknown how to produce data beyond the modular $S$- and $T$-matrices. There are also a number of subtleties in the construction which remain to be solved. In this paper, we consider an infinite family of 3-manifolds, that is, torus bundles over the circle. We show that the modular data produced by such manifolds are realized by the $\mathbb{Z}_2$-equivariantization of certain pointed premodular categories. Here the  equivariantization is performed for the $\mathbb{Z}_2$-action sending a simple (invertible) object to its inverse, also called the particle-hole symmetry. It is our hope that this extensive class of examples will shed light on how to improve the program to recover the full data of a premodular category.
\end{abstract}

\section{Introduction}
Quantum topology emerged from the discovery of the Jones polynomial \cite{jones1985polynomial} and the formulation of topological quantum field theory (TQFT) \cite{witten1989quantum,atiyah1988topological} in the 1980s. Since then, rapid progress of the subject has revealed deep connections between the algebraic/quantum world of tensor categories and the topological/classical world of 3-manifolds. One bridge connecting these two worlds is given by TQFTs. More precisely, quantum invariants of 3-manifolds and $(2+1)$-dimensional TQFTs can be constructed from modular tensor categories, a special class of tensor categories. Two fundamental families in (2+1)-dimensions are the Reshetikhin-Turaev \cite{reshetikhin1991invariants} and Turaev-Viro \cite{turaev1992state} TQFTs, both of which are based on certain tensor categories. Both families serve as vast
generalizations of the Jones polynomial to knots in arbitrary 3-manifolds. Quantum invariants induced by TQFTs provide insights to understand 3-manifolds. For example, they can distinguish some homotopically equivalent but non-homeomorphic manifolds. 

Recently, motivated by $M$-theory in physics, the authors in \cite{cho2020m} proposed another relation between tensor categories and 3-manifolds roughly in the converse direction. Explicitly, they outlined a program to construct modular tensor categories from certain classes of closed oriented 3-manifolds. A central structure to study is an $\SL(2,\bbC)$ flat connection which corresponds to a conjugacy class of morphisms from the fundamental group to $\SL(2,\bbC)$. The manifolds are required to have 
finitely many non-Abelian $\SL(2,\bbC)$ flat connections, and each must be gauge equivalent to an $\SL(2,\bbR)$ or $\SU(2)$ flat connection.  Classical invariants such as the Chern-Simons invariant and twisted  Reidemeister torsion also play a key role in the construction. 

In \cite{cui2021three}, the authors mathematically explored the program in greater detail. They systemically studied two infinite families of 3-manifolds, namely, Seifert fibered spaces with three singular fibers and torus bundles over the circle whose monodromy matrix has odd trace. It was shown that the first family realize modular tensor categories related to the Temperley-Lieb-Jones category \cite{turaev1994quantum}, and the second family was related to the quantum group category of type $B$. Based on their computations, the authors revealed several subtleties in the original proposal and made a number of insightful improvements. For instance, a simple object in the constructed category should correspond to a non-Abelian $\SL(2,\bbC)$ character from the fundamental group, rather than a conjugacy class of $\SL(2,\bbC)$ representations. Moreover, the characters are not necessarily conjugate to a $\SL(2,\bbR)$ or $\SU(2)$ character; rather, they just need to have real Chern-Simons invariants.  Also, the category to be constructed may not always be modular, but will be a pre-modular category in general, and conjecturally it is non-degenerate if and only if the first cohomology of the manifold with $\mathbb{Z}_2$ coefficients is trivial.

The efforts in \cite{cho2020m} and \cite{cui2021three} suggest a far-reaching connection between 3-manifolds and (pre)modular tensor categories. However, this program is still at its infancy, and there remain many questions to be resolved. First and foremost, the program currently only provides an algorithm to compute the modular $S$- and $T$-matrices. Other data such as the $F$-symbols and $R$-symbols, which specify the associators and braidings, respectively \cite{wang2010topological}, are still missing. Secondly, even for the modular data, the computation for the $S$-matrix essentially follows a trial-and-error procedure. A definite algorithm to achieve that is in demand. Thirdly, there are also a number of subtleties in choosing the correct set of characters as simple objects, determining the proper unit object, etc.  Before these problems can be settled, more case studies are of great value in offering insights from various perspectives, which is the motivation for the current paper. We hope the insights obtained will lead to an intrinsic understanding of how and why this program works.

In this paper, we continue the work of \cite{cui2021three} to apply the program to torus bundles over the circle with SOL geometry \cite{scott83}. The examples of Seifert fibered spaces in \cite{cui2021three} covered six of the eight geometries, the ones left being the hyperbolic and SOL. Since the program concerns closed manifolds whose Chern-Simons invariants are all real, hyperbolic manifolds are thus excluded. By definition, a torus bundle over the circle is the quotient of $T^2 \times [0,1]$ by identifying $T^2 \times \{0\}$ and $T^2 \times \{1\}$ via a self diffeomorphism of $T^2$, where $T^2$ denotes the torus. Since the mapping class group of $T^2$ is $\SL(2,\bbZ)$, a torus bundle over the circle is uniquely determined by the isotopy class of the gluing diffeomorphism, called the monodromy matrix,  which is an element in $\SL(2,\bbZ)$. 
Torus bundles whose monodromy is Anosov have SOL geometry. Equivalently, a torus bundle
has SOL geometry if and only if its monodromy matrix $A$ satisfies $|\Tr(A)| > 2$ \cite{scott83}. In \cite{cui2021three}, only special cases of $\Tr(A)$ being odd were considered, and the resulting  modular data is related to the quantum group categories of type $B$. It was conjectured that other cases correspond to this type of categories as well. However, we prove in the current paper that this is incorrect.

To state our main result, some more notations are required. For a finite Abelian group $G$ and a quadratic form $q\colon G \to \mathbb{C}$, denote by $\mcC (G,q)$ the pointed premodular category whose isomorphism classes of simple objects are $G$ and whose topological twist is given by $q$. There is a $\bbZ_2$-action on $\mcC (G,q)$ defined by sending each simple object to its dual (or its inverse viewed as a group element). This $\bbZ_2$-action is also called the particle-hole symmetry of $\mcC (G,q)$. Denote by $\mcC(G,q)^{\mathbb Z_2}$ the $\bbZ_2$-equivariantization of $\mcC (G,q)$ with respect to the particle-hole symmetry. See Section \ref{sec:equiv} for more details. The main result is as follows.
\begin{theorem*}[also see Theorem \ref{thm:main}]
For each torus bundle over the circle $M_A$ with monodromy matrix $A$, $N:= |\Tr(A)+2| $, there is an associated finite Abelian group $G_A$ isomorphic to $\bbZ_r \times \bbZ_{N/r}$  for some integer $r \ge 1$ (Lemma \ref{lem:def_G}) and a quadratic form $q_A(x) := \exp(\frac{2\pi i \tilde{q}(x)}{N}),\  \tilde{q}: G \to \bbZ_N$ (Lemma \ref{lem:well-defined}) such that the modular data realized by $M_A$ coincide with those of $\mcC(G_A,q_A)^{\mathbb Z_2}$. 
\end{theorem*}

When $r=1$, $G_A$ is a cyclic group and the particle-hole  equivariantization of $\mcC (G_A,q_A)$ is the adjoint subcategory of the metaplectic $\SO(N)_2$. When $N$ is additionally odd, we recover the result in \cite{cui2021three}. The appearance of equivariantization seems to be a salient feature of torus bundles. We leave it as a future direction to explore a possible topological interpretation of equivariantization.

We note that although the (pre-)modular categories constructed in \cite{cui2021three} and the current paper are not new in any meaningful sense, it prompts the question of what classes of 3-manifolds correspond to what classes of premodular categories. Given our limited knowledge of the program, it is difficult to construct any new categories for the moment. Our current case studies thus serve the purpose of obtaining more insight to develop the program. Eventually, the hope is to have a better understanding of the interplay between 3-manifolds and tensor categories, and to produce more interesting pre-modular categories from topology.

The rest of the paper is organized as follows. In Section \ref{sec:pre}, we review some basic facts about premodular categories and recall the program of constructing (pre)modular categories from 3-manifolds. Section \ref{sec:equiv} is devoted to computing the modular data of the equivariantization of a pointed premodular category under the particle-hole symmetry. In Section \ref{sec:compare}, we state and prove the main theorem concerning the construction of premodular categories from torus bundles.

\section{Preliminaries}\label{sec:pre}

\subsection{Premodular categories}
Here we recall some basic notation and results involving premodular categories.  A more detailed treatment is given by standard reference material \cite{bakalov2001lectures,etingof2016tensor}. Unless otherwise specified, we will always be working over the base field of complex numbers.

Let $\mcC$  be a fusion category. We denote the set of isomorphism classes of simple objects of $\mcC$ by $\Irr(\mcC)=\{X_0=\1, \cdots,X_{n-1}\}$. 

We have the fusion rules given  by
$$X_{i} \otimes X_{j} \cong \sum_{k} N_{i, j}^{k} X_{k},$$
where $N_{i,j}^k=\dim \Hom(X_i\otimes X_j,X_k)$ are called the fusion coefficients.
For any $X_i\in \Irr(\mcC)$, the fusion matrix $N_i$ is given by $ (N_i)_{k,j} = N_{i,j}^k$. The largest positive eigenvalue of $N_i$ is called the Frobenius-Perron dimension (or FP-dimension) of $V_i$ and is denoted by $\FPdim(V_i)$ (cf. \cite{ENO}). A simple object $V\in \Irr(\mcC)$ is called invertible if $\FPdim(V)=1$.

A premodular category is a braided fusion category equipped with a ribbon structure.  A ribbon structure on a braided monoidal category $\mathcal{C}$ is a family of natural isomorphisms
$\theta_{V}: V \rightarrow V$ satisfying 
$$\theta_{V \otimes W}=\left(\theta_{V} \otimes \theta_{W}\right) \circ c_{W, V} \circ c_{V, W}$$
$$ \theta_{V^{*}}=\theta_{V}^{*}$$ for all $V, W\in \mcC$, where $c$ is the braiding.

Let $\mathcal C$ be a premodular category.  The (unnormalized) $S$-matrix of  $\mathcal{C}$ has entries
$$S_{i,j}:=\tr_{X^\ast_i \otimes X_j}\left(c_{X_j, X^\ast_i} \circ c_{X^\ast _i, X_j}\right), X_i, X_j \in \Irr(\mathcal{C}).$$
 A premodular category is said to be modular if its $S$-matrix is nondegenerate.
The numbers $d_i=S_{i,0}$ are called the quantum dimensions of the simple objects $X_i\in \Irr(\mcC)$. The sum $D^2=\sum_{i=0}^{n-1}d_i^2$ is called the global dimension of $\mcC$.

As the ribbon isomorphism $\theta_{X_{i}}$ is an element of $\End(X_i)$ for any $X_i\in \Irr (\mcC)$, we can write $\theta_{X_i}$ as a scalar $\theta_i$ times the identity map on $X_i$. We call $\theta_i$ the twist of the simple object $X_i$. The $T$-matrix for a premodular $\mathcal C$ is defined to be the diagonal matrix with entries
$$T_{i,j}=\theta_i\delta_{i,j},$$
where $\theta_i$ is the twist of $X_i\in \Irr(\mcC)$.
Note the fusion coefficients and entries of $S$ and $T$ satisfy the following balancing equation
\begin{equation}\label{equ:balancing}
    \theta_{i} \theta_{j} S_{i j}=\sum_{k} N_{i^{*} j}^{k} d_{k} \theta_{k},
\end{equation}
where $i^*$ is the dual of $i$.

Given a fusion category $\mcC$, let $\mcC_{\pt}$ denote the full fusion subcategory generated by the invertible objects in $\mcC$. A fusion category $\mcC$ is said to be pointed if $\mcC=\mcC_{\pt}$. Every pointed fusion category is equivalent to  $\Vec_{G}^{\omega}$, which is  the category of finite dimensional vector spaces graded by a finite group $G$ with the associativity given by the 3-cocycle $\omega \in Z^{3}(G, \mbbC^{\times})$.

Let $G$ be a finite Abelian group, $q : G \to \mathbb{C}^\times$ be a quadratic form,\footnote{Recall that a quadratic form on an Abelian group $G$ taking values in $B$ is a map $q : G \to B$ such that (i) $q(g^{-1}) = q(g)$ for all $g \in G$, and (ii) the symmetric function $b(g,h) := \frac{q(gh)}{q(g)q(h)}$ is bimultiplicative, i.e. $b(gh,k) = b(g, k)b(h, k)$  for all $g,h,k \in G$.} and $\chi : G \to \mathbb{C}^\times$ be a character such that $\chi^2 = 1$.  As shown in \cite{drinfeld2010braided}, there exists a pointed premodular category $\mathcal{C}(G,q,\chi)$ with the following properties: 
\begin{itemize}
    \item the simple objects of $\mathcal{C}(G,q,\chi)$ are parametrized by $G$, and the monoidal product is given by the group product;
    \item $S_{gh} = b(g,h) \chi(g) \chi(h)$, where $b$ is the bicharacter $b(g,h) := \frac{q(gh)}{q(g)q(h)}$; and
    \item $T_g = q(g) \chi(g)$.
\end{itemize}
Moreover, every pointed premodular category is equivalent to some $\mathcal{C}(G,q,\chi)$. When $\chi$ is trivial, we simply denote it as $\mathcal{C}(G,q)$.

\subsection{Equivariantization} \label{sec:equi} 
For a group $\Gamma$, let $\underline \Gamma$ be the tensor category whose objects are elements of $\Gamma$ and morphisms are identities. The tensor product is given by the multiplication of $\Gamma$. Let $\mathcal C$ be a fusion category with an action of $\Gamma$ on $\mathcal C$ given by the tensor functor  $T : \underline \Gamma \to \Aut_\otimes(\mathcal{C}); \ g\mapsto T_g$.  For any $g, h \in \Gamma$ let $\nu_{g, h}$ be the isomorphism $T_{g} \circ T_{h} \simeq T_{g h}$ that defines the tensor structure on the functor $T$. A $\Gamma$-equivariant object is a pair $(X, u)$, where $X\in \mathcal C$  and $u=\left\{u_{g}: T_{g}(X) \stackrel{\sim}{\rightarrow} X \mid g \in \Gamma\right\}$, such that $u_{g h} \circ \nu_{g, h}=u_{g} \circ T_{g}\left(u_{h}\right)$ for all $g, h\in \Gamma$. The morphisms between equivariant objects are morphisms in $\mathcal{C}$ commuting with $u_g$ for all  $g\in \Gamma$.  More explicitly, a morphism between equivariant objects $(X,u) \to (Y,v)$ consists of a morphism $f: X \to Y$ such that $f \circ u_{g}=v_{g} \circ T_{g}(f)$. The category of $\Gamma$-equivariant objects of $\mathcal{C}$, which is denoted by $\mathcal{C}^{\Gamma}$, is called  the $\Gamma$-equivariantization of $\mathcal{C}$ \cite{muger2000galois,bruguieres2000categories,drinfeld2010braided}. $\mathcal C^\Gamma$ is a fusion category with the tensor product given by $(X, u) \otimes(Y, w):=(X \otimes Y, u \otimes w)$, where $(u \otimes w)_{g}:=(u_{g} \otimes w_{g}) \circ(\mu_{X, Y}^{g})^{-1}$ and $\mu_{X,Y}^g: T_g(X)\otimes T_g(Y)\to T_g(X\otimes Y)$ is the tensorator for $T$.

\subsection{A program to construct premodular categories from three manifolds}
We first recall two key ingredients, the Chern-Simons invariant and the Reidemeister torsion, that are used in the construction. 

\subsubsection{Chern-Simons invariant}

Let $X$ be a closed oriented 3-manifold and $\rho:\pi_1(X)\longrightarrow \SL(2,\mathbb{C})$ be a group morphism. Denote by $A_{\rho} $ the corresponding Lie algebra $\mathfrak{sl}(2,\mathbb{C})$-valued 1-form  on $X$. The Chern-Simons (CS) invariant of $\rho$ is defined as
\begin{equation}\label{equ:CS_integral_def}
    \CS(\rho)=\frac{1}{8\pi^2}\int_{X}\Tr(dA_{\rho}\wedge A_{\rho}+\frac{2}{3}A_{\rho}\wedge A_{\rho}\wedge A_{\rho}) \mod 1.
\end{equation}
It is a basic property that $\CS(\rho)$ only depends on the character induced by $\rho$. We will use this fact below implicitly. It is in general very difficult to compute the CS invariant directly using the integral definition. Various techniques are developed in the literature for the calculations. See for instance, \cite{auckly94}, \cite{kirk93}. Usually, the procedure involves cutting the manifold into simpler pieces, computing the CS invariant for each piece, and inferring the CS invariant of the target manifold from that of the pieces.

\subsubsection{Adjoint Reidemeister torsion}
We first recall some basics about  the Reidemeister torsion ($R$-torsion). For more details, please refer to, e.g., \cite{milnor66} and \cite{turaev01}.

Let 
$$C_*=(0\longrightarrow C_n\stackrel{\partial_n}{\longrightarrow}C_{n-1}\stackrel{\partial_{n-1}}{\longrightarrow}\ \cdots\ \stackrel{\partial_1}{\longrightarrow}C_0\longrightarrow0)$$
be a
chain complex of finite dimensional vector spaces over the field $\mathbb{C}$. Choose a basis $c_i$ of $C_i$ and a basis $h_i$ of the $i$-th homology group $H_i(C_*)$. The torsion of $C_*$ with respect to
these choices of bases is defined as follows. For each $i$, let $b_i$ be a set of vectors in $C_i$ such that $\partial_i(b_i)$ is a basis of $\Ima(\partial_i)$ and let $\tilde{h}_i$ denote a lift of $h_i$ in $\Kernel(\partial_i)$. Then the
set of vectors $\tilde{b}_i := \partial_{i+1}(b_{i+1}) \sqcup \tilde{h}_i \sqcup b_i$ is a basis of $C_i$. Let $D_i$ be the transition matrix from $c_i$ to $\tilde{b}_i$. To be specific, each column of $D_i$ corresponds to a vector in $\tilde{b}_i$ being expressed as a linear combination of vectors in $c_i$.  
Define the torsion
$$\tau(C_*,c_*,h_*):=\left|\prod_{i=0}^n \ \text{det}(D_i)^{(-1)^{i+1}}\right|
$$

We remark that the torsion does not depend on the choice of $b_i$ and the lifting of $h_i$. Also, we define the torsion as the absolute value of the usual torsion in the literature, and thus we do not need to deal with sign ambiguities.

Let $X$ be a finite CW-complex and $(V,\rho)$ be a homomorphism $\rho:\pi_1(X)\longrightarrow \SL(V)$ for some vector space $V$. Then
 $V$ turns into a left $\mathbb{Z}[\pi_1(X)]$-module via $\rho$. The universal cover $\tilde{X}$ has a natural CW structure from $X$, and its chain complex $C_*(\tilde{X})$ is a free \textit{left} $\mathbb{Z}[\pi_1(X)]$-module
via the action of $\pi_1(X)$ as covering transformations. View $C_*(\tilde{X})$ as a \textit{right} $\mathbb{Z}[\pi_1(X)]$-module by $\sigma.g := g^{-1}.\sigma$ for $\sigma \in C_*(\tilde{X})$ and $g \in \pi_1(X)$.
 We define the twisted chain complex $C_*(X;\rho):= C_*(\tilde{X})\otimes_{\mathbb{Z}[\pi_1(X)]}V$.
Let $\{e_{\alpha}^i\}_{\alpha}$ be the set of $i$-cells of $X$ ordered in an arbitrary way.
Choose a lifting $\tilde{e}_{\alpha}^i$ of $e_{\alpha}^i$ in $\tilde{X}$. It follows that $C_i(\tilde{X})$ is generated by $\{\tilde{e}_{\alpha}^i\}_{\alpha}$ as a free $\mathbb{Z}[\pi_1(X)]$-module (left or right). Choose a basis $\{v_{\gamma}\}_{\gamma}$ of $V$. Then $c_i(\rho):= \{\tilde{e}_{\alpha}^i \otimes v_{\gamma}\}$ is a $\bbC$-basis of $C_i(X;\rho)$. 
\begin{definition}\label{def:adj_cyclic}
Let $\rho:\pi_1(X)\longrightarrow \SL(V)$ be a representation.
\begin{enumerate}
    \item  We call $\rho$ acyclic if $C_*(X;\rho)$ is acyclic. Assume $\rho$ is acyclic. The torsion of $X$ twisted by $\rho$ is defined to be,
    \begin{equation*}
        \tau(X;\rho):= \tau\biggl(C_*(X;\rho),\, c_*(\rho)\biggr).
    \end{equation*}
    \item Let $\Adj: \SL(V) \to \SL(\mathfrak{sl}(V))$ be the adjoint representation of $\SL(V)$ on its Lie algebra $\mathfrak{sl}(V)$. We call $\rho$ \textit{adjoint acyclic} if $\Adj \circ \rho$ is acyclic. Assume $\rho$ is adjoint acyclic.  Define the \textit{adjoint Reidemeister torsion} of $\rho$ to be,
    \begin{equation*}
       \Tor(\rho):= \Tor(X;\rho):= \tau(X; \Adj \circ \rho).
    \end{equation*}
\end{enumerate}
\end{definition}
We remark that $\tau(X;\rho)$ is independent of the choices made for the liftings $\tilde{e}^i_{\alpha}$ of $e^i_{\alpha}$ and for the basis $\{v_{\gamma}\}$ of $V$. In this paper, we will only deal with the adjoint Reidemeister torsion $\rho$. In the sequel, we simply call it the torsion of $\rho$ if there is no potential confusion.

\subsubsection{Constructing modular data from 3-manifolds}\label{sec:contruct}
In this subsection, we briefly review the construction of modular data from 3-manifolds explained in \cite{cho2020m} and \cite{cui2021three}.  We refer the readers to \cite{cui2021three} for more detailed discussions. 

\begin{definition}
Let $X$ be a closed oriented 3-manifold, and let $\chi$ be an $\SL(2,\bbC)$-character of $X$.
\begin{itemize}
    \item $\chi$ is non-Abelian if at least one representation $\rho: \pi_1(X)\rightarrow \SL(2,\bbC)$  with character $\chi$ is non-Abelian, i.e. $\rho$ has non-Abelian image in $\SL(2,\bbC)$. 
    \item A non-Abelian character $\chi$ is {\it adjoint-acyclic} if  all non-Abelian representations $\rho: \pi_1(X)\rightarrow \SL(2,\bbC)$ with character $\chi$ are adjoint-acyclic and have the same adjoint Reidemeister torsion.
    \item A candidate label set $L(X)$  is a finite set of adjoint-acyclic non-Abelian $\SL(2,\bbC)$ characters of $X$ with a pre-chosen character $\chi_0$ such that $\CS(\chi) - \CS(\chi_0)$ is a rational number for any $\chi \in L(X)$. 
\end{itemize}
\end{definition}

We remark that for the torus bundles $M$ over the circle to be considered in Section \ref{sec:compare}, there are only finitely many non-Abelian characters and all of them are adjoint-acyclic. We always choose $L(M)$ to be the set of all non-Abelian characters.  

The potential premodular category corresponding to $X$ has $L(X)$ as the set of isomorphism classes of simple objects. Each character in $L(X)$ is a simple object type, and the pre-chosen one $\chi_0$ is the tensor unit.

The CS and torsion invariants are both well defined for characters in $L(X)$ (the latter being by definition since the label set consists of adjoint-acyclic characters). They are related to the twists and quantum dimensions. Specifically, denote by $\theta_\alpha$ the twist, and by $d_\alpha$ the quantum dimension, of $\chi_\alpha \in L(X)$. Also denote by $D^2$ the total dimension squared of the potential premodular category. Then
\begin{align}
         \theta_\alpha &=e^{-2\pi i (\CS(\chi_\alpha)-\CS(\chi_0))},\label{equ:CS_is_twist}\\
         D^2&=2\Tor(\chi_0)\label{equ:torsion0_is_D},\\
         d_\alpha^2 &=\frac{D^2}{2 \Tor(\chi_\alpha)}.\label{equ:torsion_is_qdim}
\end{align}

Apparently, for the above structures to be realized by a genuine premodular category, there must be some constraints on the label set. See \cite{cui2021three} for the definition of an admissible label set. Below, we will always assume $L(X)$ is admissible.

To define the $S$-matrix, we first need to introduce the notion of loop operators.
\begin{definition}
A primitive loop operator of $X$ is a pair $(a,R)$, where $a$ is a conjugacy class of the fundamental group $\pi_1(X)$ of $X$ and $R$ a finite dimensional irreducible representation of $\SL(2,\bbC)$.
\end{definition}

Given an $\SL(2,\bbC)$-representation $\rho$ of $\pi_1(X)$ and a primitive loop operator $(a,R)$, the weight of the loop operator $(a,R)$ with respect to $\rho$ is $W_{\rho}(a,R):=\Tr_R(\rho(a))$. It can be shown that $W_{\rho}(a,R)$ only depends on the character of $\rho$ for a fixed
choice of a primitive loop operator $(a, R)$. Hence, for an $\SL(2,\bbC)$-character $\chi$, we define $W_{\chi}(a,R)$ to be $W_{\rho}(a,R)$ for any $\rho$ representing $\chi$.

It is assumed that each simple object type  $\chi_{\alpha}$  corresponds to a finite collection of primitive loop operators
\begin{equation}\label{equ:local_operators}
    \chi_{\alpha} \mapsto \{(a^{\kappa}_{\alpha}, R^{\kappa}_{\alpha})\}_{\kappa},
\end{equation}
where $\kappa$ indexes the different primitive loop operators corresponding to $\chi_{\alpha}$. 
For a choice of $\epsilon = \pm 1$, we define the $W$-symbols
\begin{equation}\label{equ:Wbetaalpha}
    W_{\beta}(\alpha)\  :=\  \prod_{\kappa} W_{\epsilon\, \chi_{\beta}}(a_{\alpha}^{\kappa},R_{\alpha}^{\kappa}), \quad \chi_{\alpha}, \chi_{\beta} \in L(X).
\end{equation}
That is, $W_{\beta}(\alpha)$ is the product of the weight of $(a_{\alpha}^{\kappa},R_{\alpha}^{\kappa})$ with respect to the character $\epsilon \chi_{\beta}$. Here the product is taken over all primitive loop operators corresponding to $\chi_{\alpha}$, and $\epsilon \chi_{\beta}$ is the character obtained by multiplying the sign $\epsilon$ to $\chi_{\beta}$. 
The $W$-symbols and the unnormalized $S$-matrix 
are related by
\begin{equation}\label{equ:SW_relation}
    W_{\beta}(\alpha) \ = \  \frac{{S}_{\alpha\beta}}{{S}_{0\beta}} \quad \text{or} \quad{S}_{\alpha\beta}\  = \  W_{\beta}(\alpha)W_{0}(\beta),
\end{equation}
where $0$ denotes the tensor unit $\chi_0$. In particular, the quantum dimension is given by the equation
\begin{equation}\label{equ:d_is_W0}
    d_{\alpha} = W_{0}(\alpha)
\end{equation}

Unfortunately,  we do not yet know how to algorithmically define the correspondence between simple objects and loop operators, as well as the choice of $\epsilon$. Both currently involve a trial-and-error procedure. We try to guess a form of the correspondence, and check whether the resulting $S$-matrix  is consistent with other data such as the twists and the quantum dimensions obtained in Equations \ref{equ:CS_is_twist} and \ref{equ:torsion_is_qdim}. In particular, using loop operators we can compute the quantum dimension of simple objects (from the first row of the $S$-matrix). On the other hand, Equation \ref{equ:torsion_is_qdim} gives the absolute value of quantum dimension in terms of adjoint Reidemeister torsion. These two ways of computing quantum dimension place some constraints on loop operators. In practice, those constraints are sufficient to obtain loop operators.

\section{Equivariantization of particle-hole symmetry}\label{sec:equiv}
 Let $\mathcal{C}(G, q)$ denote the premodular category associated to a finite Abelian group $G$ and a quadratic form $q : G \to \mathbb{C}$ as defined in \cite{drinfeld2010braided}. In this section, we consider the  $\mathbb Z_2$-equivariantization $\mcC(G,q)^{\mathbb Z_2}$ of this premodular category, where the action $\underline{\mbbZ_2} \to \Aut_\otimes(\mcC(G,q))$ corresponds to the involution $g \mapsto-g$ in $G$. Commonly referred to as the ``particle-hole symmetry,'' this action previously appeared in the classification of metaplectic modular categories 
 \cite{ardonne2016classification,bruillard2019modular,bruillard2020dimension} 
 and equivariantization of Tambara-Yamagami categories  \cite{gelaki2009centers}.  It is clear that this action preserves the braiding as well since any quadratic form is invariant under inversion of its argument, and for any braided pointed fusion category $\mcC(G,q)$ the braiding is given by the bilinear form associated to $q$.

\begin{prop}\label{prop:fusion}
As a fusion category, $\mcC (G,q)^{\mbbZ_2}$  has the following simple objects:
\begin{itemize}
    \item[] Invertible objects: $X^+_b$,  $X^-_b$, for each $b\in G$ such that $b=-b$.

    \item[] Two-dimensional objects: $Y_{\{a,-a\}}$  for each $a\in G$ such that $a\neq -a$.
\end{itemize}
For simplicity, we denote $Y_a:=Y_{\{a,-a\}}$, and hence $Y_a = Y_{-a}$.     

The fusion rules of $\mcC (G,q)^{\mbbZ_2}$  are given by

\begin{itemize}
    \item[]  $X_b^{\epsilon}\otimes X_{b'}^{\epsilon'}\cong X_{b+b'}^{\epsilon\epsilon'}$,
    \item[] $X_b^{\epsilon}\otimes Y_a\cong Y_{a+b}$, 
    \item[] $Y_a\otimes Y_{a'}\cong \left\{\begin{array}{ll}
 X_0^+\oplus X_0^- \oplus  Y_{2a},&  \text{ if } a=\pm a',\\
Y_{a+a'} \oplus Y_{a-a'}, &  \text{ if }  a\neq\pm a',
\end{array}\right.$ where $\epsilon, \epsilon' =\pm 1$.
\end{itemize}
\end{prop}
\begin{proof}
In the notation of Section~\ref{sec:equiv}, we can pick the following representatives for each isomorphism class of simple objects: 
$X_{g}^{\pm}$ is given by $\left(g, u^{\pm}\right)$, where $u_{\varepsilon}^{\pm}: g \rightarrow g$ is given by $u_{\varepsilon}^{\pm}=(\pm 1)^{\varepsilon} \mathrm{id}_{g}$ for every $\varepsilon \in \mathbb{Z}_{2}$. Similarly, for all $g \neq-g$, there is a $\mathbb{Z}_{2}$-equivariant object $Y_{g}$ given by $\left(g \oplus-g, u\right)$, where $u_{0}: g \oplus-g \rightarrow g \oplus-g$ is given by
$$
\left(\begin{array}{cc}
\mathrm{id}_{g} & 0 \\
0 & \mathrm{id}_{-g}
\end{array}\right) \text {, }
$$
while $u_{1}:-g \oplus g \rightarrow g \oplus-g$ is given by
$$
\left(\begin{array}{cc}
0 & \mathrm{id}_{g} \\
\mathrm{id}_{-g} & 0
\end{array}\right).
$$

To see that these objects are simple, one can easily check that their endomorphism rings are one-dimensional.  For example, if $f: Y_g \to Y_g$ is a $\mathbb{Z}_2$-equivariant morphism, then $f = x\, \mathrm{id}_g \oplus y \,\mathrm{id}_{-g}$ and $f \circ u_1 = u_1 \circ T_1(f) = u_1 \circ (x\, \mathrm{id}_{-g} \oplus y\, \mathrm{id}_{g}) = y\, \mathrm{id}_{g} \oplus x\, \mathrm{id}_{-g} $.  This implies $x = y$.

These simple objects are clearly pairwise non-isomorphic (except $Y_a = Y_{-a}$ as mentioned in the statement of the theorem), and the fusion rules follow from a simple calculation.  To see that they form a complete set of representatives, one can compare the sum of the squares of their Frobenius-Perron dimensions with the categorical dimension of $\mcC (G,q)^{\mbbZ_2}$, which must be twice that of $\mcC (G,q)$ by \cite[Prop.~7.21.15]{etingof2016tensor}.
\end{proof}

Table~\ref{tab:simpleobj} goes into more detail in the special case  $G=\mathbb{Z}_r \times  \mathbb{Z}_{N/r}$.

\begin{table}[h!]
\renewcommand{\arraystretch}{1.2}
\begin{tabular}{|l|c|c|c|c|}
\hline
$(r, \frac{N}{r})$ &
  $X^\pm_{(a,b)}$ &
   $\big|\Irr\big(\mathcal{C}\left(\mathbb{Z}_{r} \times \mathbb{Z}_{N / r}, q\right)^{\mathbb{Z}_{2}}_{\text{pt}}\big)\big|$ &
  $Y_{(a, b)}$ &
  Number of $Y_{(a, b)}$ \\ \hline
$(o,o)$ &
  $(a,b)\in\langle(0,0)\rangle$ &
  $2$ &
  \begin{tabular}[c]{@{}c@{}}$a=1, \cdots, \frac{r-1}{2}$,\\ $b=1,\cdots, \frac{N/r-1}{2}$\end{tabular} &
  $\frac{N-1}{2}$ \\ \hline
$(o,e)$ &
  $(a,b)\in\langle(0, \frac{N}{2r})\rangle$ &
  $4$ &
  \begin{tabular}[c]{@{}c@{}}$a=1, \cdots, \frac{r-1}{2}$, \\ $b=1,\cdots, \frac{N}{2r}-1$\end{tabular} &
  $\frac{N}{2}-1$ \\ \hline
$(e,o)$ &
  $(a,b)\in\langle(\frac{r}{2},0)\rangle$ &
  $4$ &
  \begin{tabular}[c]{@{}c@{}}$a=1, \cdots, \frac{r}{2}-1$, \\ $b=1,\cdots, \frac{N/r-1}{2}$\end{tabular} &
  $\frac{N}{2}-1$ \\ \hline
$(e,e)$ &
  $(a,b)\in\langle(\frac{r}{2},0),(0,\frac{N}{2r})\rangle$ &
  $8$ &
  \begin{tabular}[c]{@{}c@{}}$a =1, \cdots, \frac{r}{2}-1$, \\ $b=1,\cdots, \frac{N}{2r}-1$\end{tabular} &
  $\frac{N}{2}-2$ \\ \hline
\end{tabular}
\caption{Simple objects for $\mcC(\mbbZ_r\times\mbbZ_{N/r} ,q)^{\mbbZ_2}$.  In the first column, we use `e' to denote `even' and `o' for `odd'.}
\label{tab:simpleobj}
\end{table}

\subsection{$S$- and $T$- matrices in a special case}\label{subsec:SandT}
We now specialize to the case that the minimal number of generators for $G$ is at most 2. Fixing a surjective homomorphism $\mathbb{Z} \times \mathbb{Z} \to G$, we further assume the existence of a well-defined quadratic form $\tq:G\to \mbbZ_N$ given by 
\begin{equation}
    \tq(x_1, x_2) = c_1 x_1^2+c_2 x_1x_2+c_3 x_2^2
\end{equation}
for some $c_1, c_2, c_3\in \mbbZ$ and independent of the choice of representative $(x_1, x_2) \in \mathbb{Z} \times \mathbb{Z}$.   We denote the associated bilinear form by $\lambda$, where $
\lambda: G\times G\mapsto \mathbb Z_N$ defined by $\lambda(x,y)=\tilde q(x+y)-\tilde q(x)-\tilde q(y)$, where $x=(x_1,x_2), y=(y_1,y_2)\in G$. Thus $\lambda$ can be expressed explicitly as
\begin{equation}
    \lambda(x,y)=2 c_1 x_1 y_1 + c_2(x_1 y_2+x_2 y_1)+ 2 c_3 x_2 y_2.
\end{equation}

In this case, we consider the pointed premodular category $\mathcal C(G,q)$ where $q$ is a quadratic form $q:G\to U(1)$ defined by $q=\exp{\frac{2\pi i \tilde q}{N}}$.  Let $F: \mathcal{C}(G,q)^{\mathbb{Z}_2} \to \mathcal{C}(G,q) $ be the forgetful functor.    We can equip the fusion category $\mathcal C(G,q)^{\mathbb{Z}_2}$ defined in the previous section with a premodular structure as follows. We define the braiding $c_{X,Y}$ in $\mathcal{C}(G,q)^{\mathbb{Z}_2}$ by $c_{X,Y} = c_{F(X), F(Y)}$.   Similarly, we define $\theta_X$ for $X \in \mathcal{C}(G,q)^{\mathbb{Z}_2}$ by $\theta_X = \theta_{F(X)}$.  

Combining the twists with the fusion rules described in Proposition \ref{prop:fusion}, we compute the corresponding $S$-matrix using the balancing equation \ref{equ:balancing}:

\begin{itemize}
    \item $S_{X_{(a,b)}^{\pm}, X_{(a', b')}^{\pm}}=\exp{\left(\dfrac{2\pi i }{N}\lambda(a,b,a', b')\right)}$;
    \item $S_{X_{(a,b)}^{\pm}, Y_{(a', b')}}=2\exp{\left(\dfrac{2\pi i }{N}\lambda(a,b,a', b')\right)}$;
    \item $S_{Y_{(a, b)}, Y_{(a', b')}} = 4\cos\left(\dfrac{2\pi}{N}\lambda\left(a,b,a',b'\right)\right)$.
\end{itemize}

\section{Premodular categories from SOL geometry}\label{sec:compare}
In this section, we consider a class of 3-manifolds with SOL geometry. Let $M$ be a torus bundle over the circle $S^1$ with the monodromy map $A=\begin{pmatrix}a&b\\c&d\end{pmatrix}\in \SL(2,\mathbb{Z})$. That is, $M$ is obtained from the product of the torus $T^2$ with the interval $[0,1]$ by identifying the top and the bottom tori via a self diffeomorphism $A$. It is known that $M$ has the SOL geometry if and only if  $|a+d|>2$ which is to be assumed below. We first provide the character variety of $M$, and then show that the modular data produced from $M$ is realized by the $\mathbb{Z}_2$ equivariantization of some pointed premodular categories. Throughout this section, set $N = |a+d+2| > 0$. 

\subsection{Character variety of torus bundles}\label{subsec:character}
Let $M$ be a torus bundle as above. Its fundamental group has the presentation
\begin{equation}\label{equ:torus_pi1}
\pi_1(M)=\langle x,y,h\ |\ x^ay^c=h^{-1}xh,\ x^by^d=h^{-1}yh,\ xyx^{-1}y^{-1}=1 \rangle,
\end{equation}
where $x$ and $y$ are the meridian and longitude of the torus, respectively, and $h$ corresponds to a loop around the $S^1$ component. This presentation
of $\pi_1(M)$ can be deduced directly from the standard CW-complex structure
on M. We consider non-Abelian characters of representations $ \rho: \pi_1(M) \to \SL(2,\bbC)$. According to \cite{cui2021three}, a representation realizing each character is described as follows.

The irreducible representations are given by
\begin{equation}\label{equ:torus_irrep_form}
    x \mapsto
    \begin{pmatrix}
    e^{\frac{2 \pi i k}{N}} &0\\
    0&e^{-\frac{2 \pi i k}{N}}
    \end{pmatrix},\quad 
    y \mapsto
    \begin{pmatrix}
   e^{\frac{2 \pi i l}{N}}&0\\
    0&e^{-\frac{2 \pi i l}{N}}
    \end{pmatrix},\quad
    h \mapsto
    \begin{pmatrix}
    0&1\\
    -1&0
    \end{pmatrix},
\end{equation}
where $ \Im(e^{\frac{2\pi ik}{N}})\ge 0$ and either $ e^{\frac{2 \pi i k}{N}} \neq \pm 1$ or $ e^{\frac{2 \pi i l}{N}}\neq \pm 1$ and the following equations hold, \begin{equation}\label{equ:torus_rep_equation2}
\begin{split}
    (a+1)\,k + c\, l &= \mu N\\
    b\, k + (d+1)\,l &= \nu N
\end{split}
\end{equation}
for some integers $\mu$ and $\nu$.   Since the  coefficient matrix for Equation \ref{equ:torus_rep_equation2} is nonsingular (its determinant is $\pm N$), each irreducible representation is determined by the pair $(\mu,\nu)$ and hence denoted $Y(\mu,\nu)$.

The reducible representations are of the form
\begin{equation}\label{equ:torus_reducible_form}
    x \mapsto (-1)^{\epsilon_x}
    \begin{pmatrix}
    1&1\\
    0&1
    \end{pmatrix},\quad
    y \mapsto (-1)^{\epsilon_y}
    \begin{pmatrix}
    1&u\\
    0&1
    \end{pmatrix},\quad
    h \mapsto
    \begin{pmatrix}
    v&0\\
    0&v^{-1}
    \end{pmatrix},
\end{equation}
where $\epsilon_x, \epsilon_y \in \{0,1\}$, $u \neq 0$ and
\begin{equation}\label{equ:torus_rep_equation5}
\begin{split}
    (v+v^{-1})^2 = a+d+2, \quad u = \frac{v^{-2}-a}{c}.
\end{split}
\end{equation}

Let $P$ be the quadruple that records the parity of the entries $(a,d;b,c)$ and we use $`e\textrm'$ to denote for `even' and $`o\textrm'$ for `odd'. For instance, $P = (e,e;o,e)$ means $b$ is odd and the rest are even. We then have the following possible values for $\epsilon_x$ and $\epsilon_y$ in each case:
\begin{itemize}
    \item $\epsilon_x = 0, \ \epsilon_y = 0$, with no restrictions on $P$;
    \item $\epsilon_x = 1, \ \epsilon_y = 1$, only if $P = (e,e;o,o)$ or $P= (o,o; e,e)$;
    \item $\epsilon_x = 0, \ \epsilon_y = 1$, only if $P = (o,o;o,e)$ or $P= (o,o; e,e)$;
    \item $\epsilon_x = 1, \ \epsilon_y = 0$, only if $P = (o,o;e,o)$ or $P= (o,o; e,e)$.
\end{itemize}

 We can also refer to pairs $(\epsilon_x, \epsilon_y)$ in $(\mu, \nu)$-coordinates using Equation~\ref{equ:torus_rep_equation2} and defining $k = \epsilon_x (N/2)$ and $l = \epsilon_y (N/2)$.    From Equation \ref{equ:torus_rep_equation5}, we see that for each fixed $\epsilon_x$ and $\epsilon_y$, there are four inequivalent representations but only two characters, which we denote by $X^{\pm}(\mu, \nu)$.

The torsions and Chern-Simons invariants are explicitly computed in \cite{cui2021three}. In particular, we have 
\begin{equation}\label{eq:torsion}
   \Tor(\rho)=\begin{cases}\dfrac{|a+d+2|}{4}, & \rho \text{ is irreducible}  \\|a+d+2|,& \rho \text{ is reducible} \end{cases} 
\end{equation}
and
\begin{equation}\label{eq:CS}
    \CS(\rho)=\begin{cases}\dfrac{k\nu-l\mu}{N} &\rho \text{ is irreducible}
\\
\dfrac{(a+d+2) \epsilon_{x} \epsilon_{y}+b \epsilon_{x}+c \epsilon_{y}}{4}&\rho\text{ is reducible}
\end{cases}
\end{equation}

\subsection{Solution space}
We consider solutions $(k,l)$ of Equation \ref{equ:torus_rep_equation2} in $\mathbb{Z}_N\times\mathbb{Z}_N$. Note that, for now we do not place any additional restrictions on the solutions. We denote this solution space by $G$. 

\begin{lemma}\label{lem:def_G}
$G$ is a subgroup of $\mathbb{Z}_N\times\mathbb{Z}_N$ isomorphic to $\mathbb{Z}_r\times\mathbb{Z}_{\frac{N}{r}}$, where $r=gcd(a+1, c, b, d+1)$. 
\end{lemma}
\begin{proof}
Let  $f:\mathbb{Z}\times\mathbb{Z}\to\mathbb{Z}_N\times\mathbb{Z}_N$ be the group homomorphism given by
$$f \begin{pmatrix}\mu\\\nu\end{pmatrix} =\begin{pmatrix}d+1&-c\\-b&a+1\end{pmatrix}\begin{pmatrix}\mu\\\nu\end{pmatrix}
$$
The solution space $G$ is the image of $f$ and a subgroup of $\mathbb{Z}_N\times\mathbb{Z}_N$.

Define the chain complex $\mathbb{Z}\times\mathbb{Z}\stackrel{g}{\longrightarrow}\mathbb{Z}\times\mathbb{Z}\stackrel{f}{\longrightarrow}\mathbb{Z}_N\times\mathbb{Z}_N$ where $g=\begin{pmatrix}a+1&c\\b&d+1\end{pmatrix}$. Then $\Im(f)\cong \mathbb{Z}\times\mathbb{Z}/\ker(f)$ and $\ker(f)=\Im(g)$. By considering the Smith normal form of $g$, we obtain an isomorphism $G \cong \mathbb{Z}_r\times\mathbb{Z}_{N/r}$ where $r=gcd(a+1,c,b,d+1)$.
\qed
\end{proof}
We can use $G$ to characterize non-Abelian characters of $M$ by the following lemma.
\begin{lemma}
The irreducible characters $Y(\mu,\nu)$ of $M$ are in one-to-one correspondence with subsets $\{g,-g\}\subset G$ where $2g\neq0$.   In addition, the pairs $X^{\pm}(\mu,\nu)$ of reducible non-Abelian characters are in one-to-one correspondence with elements $g\in G$ such that $2g=0$.
\end{lemma}
\begin{proof} Suppose that $(\mu,\nu)\in G$ corresponds to a representation $\rho$ as in Equation \ref{equ:torus_irrep_form} which is not necessarily non-Abelian.
We first show that $\rho$ is non-Abelian if and only if $2(\mu,\nu)\neq0$. According to the previous subsection, $\rho$ is non-Abelian if and only if $\rho(x),\rho(y)$ do not both take values in $\{I,-I\}$, which is equivalent to the statement that $\rho(x^2),\rho(y^2)$ are not both $I$. Since $2(\mu,\nu)$ corresponds to the representation $(x \mapsto \rho(x^2), y \mapsto \rho(y^2), h \mapsto \rho(h))$, the claim follows from the fact that the representation $(x \mapsto I, y \mapsto I, h \mapsto \rho(h))$ corresponds to $0\in G$.

Suppose that $(\mu_1,\nu_1),(\mu_2,\nu_2)\in G$ correspond to the same irreducible character.  Let $(k_1, l_1)$ and $(k_2, l_2)$ be the corresponding solutions to Equation \ref{equ:torus_rep_equation2}, and $\rho_1$ and $\rho_2$ be the corresponding representations as defined in Equation~\ref{equ:torus_irrep_form}. Then either $\rho_1(x)=\rho_2(x)$ and $\rho_1(y)=\rho_2(y)$, or $\rho_1(x)=\rho_2(x^{-1})$ and $\rho_1(y)=\rho_2(y^{-1})$, which implies that $(\mu_1,\nu_1)=\pm(\mu_2,\nu_2)$. This proves the first part of the lemma.

For the second part, let $\rho$ denote a reducible non-Abelian representation, and let $\epsilon_x, \epsilon_y \in \{0,1\}$ be the corresponding sign exponents as defined in Equation \ref{equ:torus_reducible_form}. By considering the diagonal entries of $\rho(x)$ and $\rho(y)$, such a representation $\rho$ exists if and only if the following equations are satisfied.
$$(a+1)\epsilon_x\frac{N}{2}+c\epsilon_y\frac{N}{2}=\mu N
$$
$$b\epsilon_x\frac{N}{2}+(d+1)\epsilon_y\frac{N}{2}=\nu N
$$
The solutions of above equations are in one-to-one correspondence with elements in $G$ of order 1 or 2. Fixing $(\epsilon_x,\epsilon_y)$, the corresponding characters occur in pairs $X^\pm(\mu,\nu)$.  This proves the second part of the lemma.
\qed
\end{proof}

\newcommand{\hq}{\hat{q}}

We now define a map $\hq:\mathbb{Z}\times\mathbb{Z}\to\mathbb{Z}_N$ by $\hq(\mu,\nu)=c\nu^2+(a-d)\mu\nu - b\mu^2$.

\begin{lemma}\label{lem:well-defined}
The map $\hq$ induces a  quadratic form $\tq : G \to \mathbb{Z}_N$.
\end{lemma}
\begin{proof}
Since $\ker(f)=\Im(g)=\left\{\begin{pmatrix}a+1&c\\b&d+1\end{pmatrix}\begin{pmatrix}i\\j\end{pmatrix}i,j\in\mathbb{Z}\right\}$, it suffices to show that $\hq(\mu+a+1, \nu+b) = \hq(\mu, \nu)$ and $\hq(\mu+c, \nu+d+1) = \hq(\mu, \nu)$ for general $\mu$ and $\nu$.  We have
\begin{align*}
    \hq(\mu+a+1, \nu+b) - \hq(\mu, \nu)  &= c(\nu + b)^2 + (a - d) (\mu + a + 1)(\nu + b)\\
    & \qquad -b (\mu + a + 1)^2  - \hq(\mu,\nu) \\
    & =  -b(d - a + 2 a + 2)\mu  + (2bc + (a-d)(a+1))\nu \\
    &  \qquad - b( -bc + (d-a)(a+1) + (a+1)^2) \\
     & = (-2 + 2ad - ad + a^2 + a - d)\nu  \\
     & \qquad - b( 1 - ad + ad - a^2 + d - a + a^2 + 2a + 1) \\
     & = (-2 + a(-a -2) + a^2 +a +2 +a)\nu  \\
    & =  0,
\end{align*}
and
\begin{align*}
\hq(\mu + c, \nu+d+1)-\hq(\mu, \nu)& =c(\nu+(d+1))^{2}+(a-d)(\mu+c)(\nu+d+1)\\
& \qquad -b(\mu+c)^{2}-\hq(\mu,\nu)\\
& =c(2 (d+1)+(a-d) ) \nu +((a-d)(d+1)-2 b c) \mu\\
& \qquad +c((d+1)^{2}+(a-d)(d+1)-b c)\\
& =(-d^{2}-d+a-a d+2)\mu +  c(d+1+a d+a-b c)\\
& = 0
\end{align*}
Thus $\hq$ induces a well defined map $\tq: G \to \mathbb{Z}_N$.  It is routine to check that this map is a quadratic form.
\qed
\end{proof}

\subsection{$S$- and $T$- matrices from torus bundles}

We define the loop operators for non-Abelian characters by
$$X^{\pm}(\mu,\nu)\mapsto (x^my^n,\Sym^0)
$$
$$Y(\mu,\nu)\mapsto(x^my^n,\Sym^1)
$$
where $m=-b\mu+(a-1)\nu$, $n=(-d+1)\mu+c\nu$, and $\Sym^j$ denotes the unique $(j+1)$-dimensional irreducible representation of $\SL(2, \mathbb{C})$. We choose $X^+(0,0)$ to correspond to the monoidal unit object. Each character can be represented by infinitely many representatives $(\mu,\nu)\in\mathbb{Z}\times\mathbb{Z}$, but as the following lemma shows, the $S$-matrix is independent of this choice.

\begin{lemma}\label{lem:smatrices}
Let $S^l$ be the $S$-matrix constructed from loop operators as above, then
$$S^l_{X^{\pm}(\mu_1,\nu_1),X^{\pm}(\mu_2,\nu_2)}=1
$$
$$S^l_{X^{\pm}(\mu_1,\nu_1),Y(\mu_2,\nu_2)}=2
$$
$$S^l_{Y(\mu_1,\nu_1),Y(\mu_2,\nu_2)}=4\cos\left(\frac{2\pi}{N}\lambda\left(\mu_1,\nu_1,\mu_2,\nu_2\right)\right)
$$
where $\lambda(\mu_1,\nu_1,\mu_2,\nu_2) = \tq(\mu_1 + \mu_2, \nu_1 + \nu_2) - \tq(\mu_1, \nu_1) - \tq(\mu_2, \nu_2)$ is the bilinear form associated to the quadratic form $\tq : G \to \mathbb{Z}_N$ defined in Lemma~\ref{lem:well-defined}. 
\end{lemma}
\begin{proof}
From Equation \ref{equ:Wbetaalpha}, we have the following $W$-symbols
$$W_{X^{\pm}(\mu_1,\nu_1)}(X^{\pm}(\mu_2,\nu_2))=W_{Y(\mu_1,\nu_1)}(X^{\pm}(\mu_2,\nu_2))=1
$$
$$W_{X^{\pm}(\mu_1,\nu_1)}(Y(\mu_2,\nu_2))=\Tr(X^{\pm}(\mu_1,\nu_1)(x^{m_2}y^{n_2}))
$$
$$W_{Y(\mu_1,\nu_1)}(Y(\mu_2,\nu_2))=\Tr(Y(\mu_1,\nu_1)(x^{m_2}y^{n_2}))
$$
Thus,
$$S^l_{X^{\pm}(\mu_1,\nu_1),X^{\pm}(\mu_2,\nu_2)}=W_{X^{\pm}(\mu_2,\nu_2)}(X^{\pm}(\mu_1,\nu_1))W_{X^{+}(0,0)}(X^{\pm}(\mu_2,\nu_2))=1
$$
$$S^l_{X^{\pm}(\mu_1,\nu_1),Y(\mu_2,\nu_2)}=W_{Y(\mu_2,\nu_2)}(X^{\pm}(\mu_1,\nu_1))W_{X^{+}(0,0)}(Y(\mu_2,\nu_2))=2
$$
$$S^l_{Y(\mu_1,\nu_1),Y(\mu_2,\nu_2)}=W_{Y(\mu_2,\nu_2)}(Y(\mu_1,\nu_1))W_{X^{+}(0,0)}(Y(\mu_2,\nu_2))=2\Tr(Y(\mu_2,\nu_2)(x^{m_1}y^{n_1}))
$$
and
\begin{align*}
\Tr(Y(\mu_2,\nu_2)(x^{m_1}y^{n_1}))&=2\cos\left(2\pi\frac{k_2m_1+l_2n_1}{N}\right)
\\
&=2\cos\left(\frac{2\pi}{N}\begin{pmatrix}m_1&n_1\end{pmatrix}\begin{pmatrix}k_2\\l_2\end{pmatrix}\right)
\\
&=2\cos\left(\frac{2\pi}{N}\begin{pmatrix}\mu_1&\nu_1\end{pmatrix}\begin{pmatrix}-b&-d+1\\a-1&c\end{pmatrix}\begin{pmatrix}d+1&-c\\-b&a+1\end{pmatrix}\begin{pmatrix}\mu_2\\\nu_2\end{pmatrix}\right)
\\
&=2\cos\left(\frac{2\pi}{N}\begin{pmatrix}\mu_1&\nu_1\end{pmatrix}\begin{pmatrix}-2b&a-d\\a-d&2c\end{pmatrix}\begin{pmatrix}\mu_2\\\nu_2\end{pmatrix}\right)
\\
&=2\cos\left(\frac{2\pi}{N}\lambda(\mu_1,\nu_1,\mu_2,\nu_2)\right).
\end{align*}
\qed
\end{proof}

Defining $q : G\to U(1)$  by $q(x)= e^{\frac{2\pi i \tq(x)}{N}}$, we have the premodular category $\mathcal{C}(G,q)$ and its $\mathbb{Z}_2$-equivariantization  $\mathcal{C}(G,q)^{\mathbb{Z}_2}$ as described in Section~\ref{sec:equiv}.  Our main theorem is the following.

\begin{theorem}\label{thm:main}
The $S$- and $T$-matrices constructed from torus bundles with Sol geometry coincide with those of the $\mathbb{Z}_2$-equivariantization  $\mathcal{C}(G,q)^{\mbbZ_2}$.
\end{theorem}
\begin{proof}
From Equations \ref{equ:torus_rep_equation2} and \ref{eq:CS}, we have $\CS(\rho) =\frac{-c\nu+(d-a)\mu\nu+b\mu^2}{N} = -\frac{\tq(\mu,\nu)}{N}$.   
Thus, the $T$-matrix of $\mathcal{C}(G,q)^{\mbbZ_2}$ as defined in Section~\ref{subsec:SandT} coincides with the one constructed directly from the torus bundle as defined in Equation \ref{equ:CS_is_twist}.

Let $S^e$ denote the $S$-matrix from the $\mathbb{Z}_2$-equivariantization $\mcC(G,q)^{\mbbZ_2}$ as defined in Section~\ref{subsec:SandT}, and let $S^l$ denote the $S$-matrix from the local operator construction as defined in Lemma~\ref{lem:smatrices}.  We first consider the following entry: $$S^e_{X^{\pm}(\mu_1,\nu_1),X^{\pm}(\mu_2,\nu_2)}=\frac{q(X(\mu_1+\mu_2,\nu_1+\nu_2))}{q(X(\mu_1,\nu_1))q(X(\mu_2,\nu_2))}
$$

When $X(\mu_1,\nu_1)=X(\mu_2,\nu_2)$, according to the group structure of $G$ we have  $X(\mu_1+\mu_2,\nu_1+\nu_2)=X(0,0)$. Thus $S^e_{X^{\pm}(\mu_1,\nu_1),X^{\pm}(\mu_2,\nu_2)}=1$.  Similarly, if  $X(\mu_i, \nu_i) = X(0,0)$ for either $i$, then clearly $S^e_{X^{\pm}(\mu_1,\nu_1),X^{\pm}(\mu_2,\nu_2)}=1$.

When $X(\mu_1,\nu_1)\neq X(\mu_2,\nu_2)$ and  $(\mu_i, \nu_i) \neq (0,0)$ for all $i$, then the characters $X(\mu_1+\mu_2,\nu_1+\nu_2)$, $X(\mu_1,\nu_1)$, and $X(\mu_2,\nu_2)$ are all distinct.  Using the notation of Section~\ref{subsec:character}, these characters must correspond to the cases $(\epsilon_x, \epsilon_y) \in \{(1,0), (0,1), (1,1)\}$.  As mentioned in that section, this can only occur if the parities of $(a,d;b,c)$ are $(o,o;e,e)$.  Using the fact that $ad - bc = 1$, one obtains that $N = a + d + 2 = 0 \pmod{4}$.  Thus Equation~\ref{eq:CS} reduces to $\CS(X(\mu, \nu)) = (b \epsilon_x + c \epsilon_y)/4$.  By inspection,  one finds that applying $q(\mu, \nu) = \exp(-2\pi i \CS(X(\mu, \nu)))$ to the $(\mu,\nu)$ corresponding to $(\epsilon_x, \epsilon_y) \in \{(1,0), (0,1), (1,1)\}$ yields either the multiset $-1,-1,1$ or $1,1,1$. Thus $S^e_{X^{\pm}(\mu_1,\nu_1),X^{\pm}(\mu_2,\nu_2)}=1$.

Next we consider 
$$S^e_{X^{\pm}(\mu_1,\nu_1),Y(\mu_2,\nu_2)}=2\frac{q(Y(\mu_1+\mu_2,\nu_1+\nu_2))}{q(X(\mu_1,\nu_1)) q(Y(\mu_2,\nu_2))}.
$$

Without loss of generality, we only need to consider two cases: $(\mu_1, \nu_1)$ corresponding to $(k_1=\frac{N}{2},l_1=0)$ where the parity of $(a,d;b,c)$ is $(o,o;e,o)$, and $(\mu_1, \nu_1)$ corresponding to $(k_1=\frac{N}{2},l_1=\frac{N}{2})$ for $(o,o;e,e)$ and $(e,e;o,o)$.

When $k_1=\frac{N}{2}$ and $l_1=0$,
\begin{align*}
S^e_{X^{\pm}(\mu_1,\nu_1),Y(\mu_2,\nu_2)}&=2\exp\left(2\pi i\frac{(k_2+\frac{N}{2})(\nu_2+\frac{b}{2})-l_2(\mu_2+\frac{a+1}{2})-k_2\nu_2+l_2\mu_2-\frac{Nb}{4}}{N}\right)
\\
&=2\exp\left(2\pi i\frac{N\nu_2+k_2b-l_2(a+1)}{2N}\right)
\\
&=2\exp\left(2\pi i\frac{N\nu_2+\nu_2N-l_2(d+1)-l_2(a+1)}{2N}\right)
\\
&=2\exp\left(2\pi i\frac{-l_2(a+d+2)}{2N}\right)
\\
&=2\exp\left(2\pi i\frac{-l_2}{2}\right)
\end{align*}
Since $l_2=-b\mu_2+(a+1)\nu_2$ and $b,a+1$ are both even, $l_2$ is even. Thus $S^e_{X^{\pm}(\mu_1,\nu_1),Y(\mu_2,\nu_2)}=2$.

When $k_1=\frac{N}{2}$ and $l_1=\frac{N}{2}$,

\begin{align*}
S^e_{X^{\pm}(\mu_1,\nu_1),Y(\mu_2,\nu_2)}&=2\exp\left(\frac{2\pi i}{N}((k_2+\frac{N}{2})(\nu_2+\frac{b+d+1}{2}) -(l_2+\frac{N}{2})(\mu_2+\frac{a+c+1}{2}) \right.\\
& \left.\qquad -k_2\nu_2+l_2\mu_2-\frac{N(a+c+b+d+2)}{4} )\right)
\\
&=2\exp\left(\frac{\pi i}{N}(N(\nu_2-\mu_2)+k_2(b+d+1)-l_2(a+c+1))\right)
\\
&=2\exp\left(\frac{\pi i}{N}(N(\nu_2-\mu_2)+N\nu_2-(d+1)l_2\right.\\
&\qquad  +k_2(d+1)-N\mu_2+k_2(a+1)-l_2(a+1))\Big)
\\
&=2\exp\left(\frac{\pi i}{N}(k_2-l_2)(a+d+2)\right)
\\
&=2\exp\left(\pi i(k_2-l_2)\right)
\end{align*}

Since $k_2-l_2=(b+d+1)\mu_2-(a+c+1)\nu_2$ and $b+d+1$, $a+c+1$ are both even, $k_2-l_2$ is even. Thus $S^e_{X^{\pm}(\mu_1,\nu_1),Y(\mu_2,\nu_2)}=2$.

Lastly, it follows from their definitions in Lemma \ref{lem:smatrices} and Section \ref{subsec:SandT} that $S^e_{Y(\mu_1,\nu_1),Y(\mu_2,\nu_2)}=S^l_{Y(\mu_1,\nu_1),Y(\mu_2,\nu_2)}$.

\end{proof}

\vspace{0.2cm}
\bibliographystyle{abbrv}
\bibliography{ref}

\end{document}